\documentclass{amsart}
\usepackage{amsfonts}
\usepackage{graphicx}
\usepackage{amscd}
\usepackage{amsmath}
\usepackage{amssymb}
\usepackage{stmaryrd} 
\usepackage{tikz}
\usepackage{pgfplots}

\makeatletter
\@namedef{subjclassname@2010}{%
\textup{2010} Mathematics Subject Classification}
\makeatother
\theoremstyle{plain}
\newtheorem*{acknowledgement}{Acknowledgement}

\newtheorem{definition}{\bf Definition}

\newtheorem{proposition}{\bf Proposition}
\newtheorem{remark}{Remark}

\newtheorem{theorem}{\bf Theorem}
\newcommand{\Limsup}{\operatorname*{Lim\,sup}}
\usepackage[colorlinks=true,plainpages=true,citecolor=blue,linkcolor=black]{hyperref}

\theoremstyle{definition}
\newtheorem{example}{\bf Example}
\numberwithin{equation}{section}

\def\tilde{\widetilde}

\def\dom{{m dom}\,}

\def\min{\mbox{\rm minimize}}

\def\B{\mathbb B}

\def\ox{\overline{x}}

\def\oz{\overline{z}}

\def\disp{\displaystyle}
\def\tto{\rightrightarrows}

\def\Tilde{\widetilde}
\def\Bar{\overline}

\def\epsilon{\varepsilon}
\def\ox{\bar{x}}

\def\oz{\bar{z}}

\def\cone{\mbox{\rm cone}\,}

\def\dom{\mbox{\rm dom}\,}

\def\dn{\downarrow}
\def\O{\Omega}
\def\ph{\varphi}

\def\lm{\lambda}
\def\gg{\gamma}
\def\dd{\delta}

\def \N{{\rm I\!N}}

\def\th{\theta}

\def\Limsup{\mathop{{\rm Lim}\,{\rm sup}}}

\def\Limsup{\mathop{{\rm Lim}\,{\rm sup}}}

\title[Article Title]{Regularized Multiobjective Optimization\\ with Directionally Lipschitzian Data}

\author{G. C. Bento$^\ast$}

\author{J. X. Cruz Neto}

\author{J. O. Lopes}

\author{B. S. Mordukhovich}

\author{P. R. Silva Filho}

\address[G. C. Bento]{Institute of Mathematics and Statistics, IME, Federal University of Goi\'as, Goi\^ania, Goi\'as, Brazil}
\email{glaydston@ufg.br}

\address[J. X. Cruz Neto]{Departamento de Matem\'{a}tica, CCN, Federal University of Piau\'{\i}\\
 Te\-re\-si\-na, Piau\'{\i}, Brazil.}
\email{jxavier@ufpi.edu.br}

\address[J. O. Lopes]{Departamento de Matem\'{a}tica, CCN, Federal University of Piau\'{\i}\\
 Te\-re\-si\-na, Piau\'{\i}, Brazil.}
\email{jurandir@ufpi.edu.br}

\address[B. S. Mordukhovich]{Department of Mathematics and Institute for AI and Data Science, Wayne State University, Detroit, Michigan, United States}
\email{boris@math.wayne.edu}

\address[P. R. Silva Filho]{Departamento de Matem\'{a}tica, CCN, Federal University of Piau\'{\i}\\
 Te\-re\-si\-na, Piau\'{\i}, Brazil.}
\email{pedrorodrigues@ufpi.edu.br}

\begin{document}

\newcommand{\spacing}[1]{\renewcommand{\baselinestretch}{#1}\large\normalsize}
\spacing{1.2}

\begin{abstract}

The paper is devoted to the study of regularized versions of multiobjective optimization problems described by directionally Lipschitzian functions. Such regularizations appear in proximal-type algorithms of multiobjective optimization, various models of machine learning, medical physics, etc. We investigate and illustrate several useful properties of directionally Lipschitzian functions, which distinguish them from locally Lipschitzian ones. By using advanced tools of variational analysis and generalized differentiation revolving around the limiting/Mordukhovich subdifferential, we derive necessary conditions for Pareto optimality in regularized multiobjective problems.\\[1ex]
\end{abstract}

\maketitle

\section{Introduction}\label{intro}

Multiobjective optimization is an area of decision-making with multiple criteria that is dedicated to simultaneously optimizing problems involving several objective functions. Since there exists no single point that can minimize all the objective functions at the same time, various notions of vector optimality have been introduced and studied in the literature. Most of them revolve around {\em Pareto optimality/efficiency}, which is often replaced by the {\em weak Pareto} notion that is more convenient for mathematical methods. Multiobjective optimization has a variety of applications to numerous fields of science and technology including engineering, economics, statistics, industry, agriculture, artificial intelligence, medical physics, etc. Among great many publications devoted to various versions of multiobjective optimization, its extensions and applications, we refer the reader to, e.g., \cite{bao,proton,gal1997,jahn2009vector,ktz,kobis,Mordukh,mordukhovich2018variational,white1990} and the bibliographies therein.

Advanced trends in optimization theory and applications, particularly to models of machine learning, statistics, and data science, consist of adding {\em regularizing terms} to minimizing (scalar or vector) costs. Such terms are proved to be crucial in the design and justification of {\em fast algorithms} in optimization and applied modeling. To this end, we mention the {\em proximal point algorithm} \cite{Rockafellar1976}, as well as its vectorial \cite{Bonnel2005,chuong2011hybrid} and various multiobjective \cite{bento2024refined,Bento2018,bento2022pareto,bento2018proximal} extensions, with the {\em proximal} (squared norm) regularization, the {\em Lasso regression model} in statistics and machine learning \cite{lasso} with the {\em $\ell^1$-norm} regularization, and the most recent {\em proton therapy} multiobjective model in cancer research with the {\em $\ell^0$-norm} regularization \cite{proton}.  

Most of the known multiobjective models in optimization theory concerned problems with smooth, convex, and/or Lipschitz continuous data and addressed deriving necessary optimality conditions for {\em weak Pareto} minimizers. The first optimality conditions for weak Pareto solutions to Lipschitzian multiobjective problems appeared in \cite{clarke,minami} being expressed in terms of Clarke's {\em generalized gradient}. 

In this paper, we consider multiobjective optimization problems defined by {\em directionally Lipschitzian} functions. Our main attention here is directed not towards general multiobjective models but towards their {\em regularizations}. For definiteness and simplicity, our study in this paper is confined to the {\em proximal regularization}, while the two other types of regularizations mentioned above can also be treated by using our variational techniques. In this way, we derive necessary optimality conditions for {\em Pareto} (not just weak Pareto) minimizers of the {\em regularized multiobjective} models in terms of Mordukhovich's {\em limiting subdifferential}, which is essentially smaller than its Clarke's counterpart and enjoys comprehensive {\em calculus rules}. 

The rest of the paper is organized as follows. Section~\ref{prel} presents some preliminaries from variational analysis and generalized differentiation used below. In Section~\ref{sec:dir-lip}, we study the class of directionally Lipschitzian functions with deriving their useful properties and constructing instructive examples. Section~\ref{secMultiobectiveoptmin} formulates and discusses regularized multiobjective problems with directionally Lipschitzian data and establishes some properties of their Pareto optimal solutions. In Section~\ref{sec3}, we obtain necessary conditions for Pareto minimizers in regularized multiobjective optimization expressed in terms of limiting subgradients. The final Section~\ref{conc} resumes the main achievements of the paper and discusses some directions of our future research and perspectives for applications.

\section{Generalized Differentiation in Variational Analysis}\label{prel}

This section presents some notions and results of variational analysis and generalized differentiation, which are mostly taken from the books \cite{Morduk,mordukhovich2018variational,rw} and play a crucial role in establishing the main achievements of this paper. 

Let $f:\mathbb{R}^n\to\Bar{\mathbb{R}}:=(-\infty,\infty]$ be an extended-real-valued lower semicontinuous (l.s.c.) function. For a sequence $\{u^k\}\subset \mathbb{R}^n$, we use the notation $u^k\stackrel{f}{\to} \bar{u}$ to signify that $u^k\to \bar{u}$ with $f(u^k)\to f(\bar{u})$. The {\em regular/Fr\'echet subdifferential} of $f$ at $\bar{x}\in\mbox{dom}f:=\{x\in\mathbb{R}^n\;|\;f(x)<\infty\}$ is defined by
\begin{equation}\label{reg-sub}
\partial^F f(\overline{x}):=\bigg\{v\in \mathbb{R}^n\;\bigg|\;\liminf_{x\to\overline{x}}\dfrac{f(x)-f(\overline{x})-\langle v,x-\overline{x}\rangle}{\|x-\overline{x}\|}\geq 0\bigg\},
\end{equation} 
while the {\em limiting/Mordukhovich subdifferential} of $f$ at $\bar{x}\in \mbox{dom}f$ is 
\begin{equation}\label{lim-sub}
\partial^Mf(\bar{x})=\Limsup_{x\stackrel{f}{\to} \bar{x}}\partial^Ff(x)=\left\{v\in \mathbb{R}^n\Big|\;\exists\, x^k\stackrel{f}{\to} \bar{x} ,\,v^k\in\partial^Ff(x^k),\,v^k\to v\right\},
\end{equation}
where ``Limsup" denotes the (Painlev\'e-Kuratowski) {\em outer limit} of a set-valued mapping/multifunction $G\colon\mathbb{R}^n\tto\mathbb{R}^m$ as $z\to\ox$ given by 
\begin{equation}\label{pk}
\Limsup_{z\to\oz} G(z):=\big\{y\in\mathbb{R}^m\;\big|\;\exists\,z^k\to\oz,\;y^k\to y\;\mbox{ with }\;y^k\in G(z^k)\big\}.
\end{equation}
It is well known that if $f$ is (Fr\'echet) differentiable at $\bar{x} \in \mbox{dom}f$, then its regular subdifferential at $\ox$ is a singleton $\partial^Ff(\bar{x}) = \{\nabla f(\bar x)\}$. If $f$ is continuously differentiable around $\bar{x} \in \mbox{dom}f$ (or merely strictly differentiable at this point), then its limiting subdifferential reduces to the same singleton. In contrast to \eqref{reg-sub}, the limiting subdifferential \eqref{lim-sub}, while being nonconvex-valued, enjoys {\em full calculus} based on {\em variational/extremal principles} of variational analysis. In particular, we have the following fundamental {\em sum rule} 
(see, e.g., \cite[Corollary~2.21]{mordukhovich2018variational}), where $\partial^\infty f(\ox)$ stands for the {\em singular/horizon subdifferential} of $f$ at $\ox\in\dom f$ defined by
\begin{equation}\label{sin-sub}
\partial^{\infty}f(\bar{x}):=\big\{v\in \mathbb{R}^n\;
\big|\;\exists\, x^k\stackrel{f}{\to} \bar{x}, \, \lm_k\downarrow 0,\,v^k\in\partial^Ff(x^k),\,\lm_kv^k\to v\big\}.
\end{equation}

\begin{proposition}\label{cor:subgradients_sum} Let $f_i\colon\mathbb{R}^n\to \Bar{\mathbb{R}}$ be extended-real-valued functions that are l.s.c.\ around some point $\ox\in\dom f_i$, $i=1,\ldots,m$. Impose the qualification condition
\begin{equation}\label{eq:qualification_condition}
\big[v_1 + \ldots + v_m = 0,\,  v_i \in \partial^\infty f_i (\bar{x})\big] \Longrightarrow\big[v_i = 0\;\mbox{ for all }\;i=1,\ldots,m\big]
\end{equation}
expressed in terms of the singular subdifferential \eqref{sin-sub}. Then we have
\begin{equation}\label{eq:subdifferential_sum}
\partial (f_1 + \ldots + f_m)(\bar{x}) \subset \partial f_1(\bar{x}) + \ldots + \partial f_m(\bar{x}).
\end{equation}
\end{proposition}
It is well known (see, e.g., \cite[Theorem~4.15]{mordukhovich2018variational})
that the condition $\partial^\infty f(\ox)=\{0\}$ is a complete {\em characterization} of the {\em local Lipschitzian} property of an l.s.c.\ function $f$ around $\ox$. Thus the qualification condition \eqref{eq:qualification_condition} is satisfied if all but one of the functions $f_i$ are Lipschitz continuous around $\ox$. Note that the singular subdifferential mapping \eqref{sin-sub} shares full calculus with the limiting one \eqref{lim-sub}. 

Another important property of the limiting subgradient mapping is its {\em robustness} by which we mean the following.

\begin{proposition}\label{weakgraph}
Let $f\colon\mathbb{R}^n\to \Bar{\mathbb{R}}$ be l.s.c. function around $\ox\in\dom f$, and let $\{x^k\},\{\omega^k\}\subset \mathbb{R}^n$ be sequences such that $\omega^k\in\partial^Mf(x^k)$ for all $k \in \mathbb{N}:=\{1,2,\ldots\}$ with $x^k\to x$, $\omega^k\to\omega$, and $f(x^k)\to f(x)$ as $k\to\infty$. Then we have the inclusion $\omega\in\partial^Mf(x)$.
\end{proposition}
The proof of this result follows directly from definition \eqref{lim-sub} and the diagonal process in finite dimensions. The robustness property also holds for the singular subdifferential \eqref{sin-sub} while fails to the regular subgradient mapping \eqref{reg-sub} even in the simplest case of $f(x):=-|x|$ with $\ox=0\in\mathbb{R}$.

The results obtained below via the limiting subdifferential \eqref{lim-sub} for Pareto minimizers have serious advantages in comparison with known in the literature even for the case of weak Pareto solutions, where the Clarke's generalized gradients of locally Lipschitzian functions are used instead of \eqref{lim-sub}. Given a function $f\colon\mathbb{R}^n\to\mathbb{R}$ that is Lipschitz continuous around $\ox$, recall first the (Clarke) {\em generalized directional derivative} of $f$ at $\bar{x}$ in the direction of $d$ defined by
\begin{align}
\nonumber f^{\circ}(\bar{x},d)=\limsup_{{x \to \bar{x}} \atop {t \downarrow 0 }} \dfrac{f(x+td)-f(x)}{t}.
\end{align}
Then (Clarke's) {\em generalized gradient} of $f$ at $\bar{x}$ is
\begin{align}\label{cl}
\partial^Cf(\bar{x})=\big\{v\in\mathbb{R}^n\;\big|\;\langle v,d\rangle\leq 
 f^{\circ}(\bar{x}, d)\;\mbox{ for all }\; d \in \mathbb{R}^n\big\}.
\end{align}
We have the following relationship between the generalized gradient and limiting subdifferentials of locally Lipschitzian functions (see \cite[Theorem~3.57]{Morduk}):
\begin{equation*}\label{clco}
\partial^C f(\bar{x})={\rm clco}\,\partial^M f(\bar{x}),
\end{equation*}
where ``clco" stands for the weak closure of the convex hull of the set. Note that, while the generalized gradient \eqref{cl} of locally Lipschitzian functions is robust, this property fails in more general settings.

Finally in this section, recall that the {\em distance function} $d_{\Omega}(x)$ associated with a nonempty set $\O\subset \mathbb{R}^n$ is defined by
\begin{equation}\label{dist}
d_{\Omega}(x):=\inf\big\{\|x-y\|\;\big|\;\;y \in \Omega\big\},\quad x\in \mathbb{R}^n.
\end{equation}
It is well-known that the distance function \eqref{dist} is globally Lipschitzian on $\mathbb{R}^n$ with Lipschitz constant one. If $\O$ is closed around $\ox\in\O$, then the limiting subdifferential of \eqref{dist} is calculated by (see \cite[Theorem~1.33(i)]{mordukhovich2018variational})
\begin{equation}\label{sub-dist}
\partial^Md_\O(\ox)=N^M_\O(\ox)\cap\B,
\end{equation}
where $\B$ is the closed unit ball of $\mathbb{R}^n$, and where the (Mordukhovich) {\em limiting normal cone} is defined by 
\begin{equation}\label{nor}
N^M_\O(\ox):=\Limsup_{x\to\ox}\big[\cone\big(x-\Pi_\O(x)\big)\big]
\end{equation}
via the outer limit \eqref{pk} involving the Euclidean projector $\Pi_\O$ associated with the set  $\O$. The notation ``cone" stands in \eqref{nor} for the (nonconvex) conic hall of a set.

\section{Directionally Lipschitzian Functions}\label{sec:dir-lip}

Now we define, following Rockafellar \cite{Roc79},
a class of extended-real-valued functions that extending the local Lipschitz continuity; see \cite{burke2021,Roc79,rw} for more details.

\begin{definition}\label{dir-lip} Let $f:\mathbb{R}^n\to\overline{\mathbb{R}}$ be l.s.c. function around $\ox\in\dom f$. We say that $f$ is {\sc directionally Lipschitzian} at $\ox\in\dom f$  if there is a unit vector $u\in \mathbb{R}^n$ with
\begin{eqnarray}\label{dir-lip1}
\displaystyle\limsup_{\substack{x\stackrel{f}{\to} \overline{x} \\ v\to u \\ t\downarrow 0}}\frac{f(x+tv)-f(x)}{t}<\infty.
\end{eqnarray}
The function \( f \) is directionally Lipschitzian on $\mathbb{R}^n$ if it satisfies \eqref{dir-lip1} at every point in \( \mathbb{R}^n \).  Furthermore, a {\sc vector mapping} \( F: \mathbb{R}^n \to \mathbb{R}^m \) is directionally Lipschitzian at $\ox$ $($resp. on $\mathbb{R}^n)$ if each of its components \( f_i: \mathbb{R}^n \to \mathbb{R} \), $i=1,\ldots,m$, enjoys the corresponding property.
\end{definition}

In what follows, we reveal some useful assertions concerning directionally Lipschitzian functions. Let us start with providing an example of a real function on $\mathbb{R}^n$,  which is continuous and directionally Lipschitzian while  not locally Lipschitzian around the reference point.

\begin{example}\label{ex1} Consider the function $f:\mathbb{R}^n \rightarrow \mathbb{R}$ defined by 
\begin{equation}\label{f-ex1}
f(x):=\sum_{i=1}^n\big({x_i}\big)^\frac{1}{3},\quad x=(x_1,\ldots,x_n)\in\mathbb{R}^n.
\end{equation}
This function is obviously continuous but not locally Lipschitzian around $\ox=0$.  We now show that \eqref{f-ex1} is directionally Lipschitzian at this point. Fix the vector
\[
u = \left(-\frac{1}{\sqrt{n}}, \ldots, -\frac{1}{\sqrt{n}}\right).
\]
and take \( v = (v_1, \ldots, v_n) \in \mathbb{R}^n \) such that \( v \to u \). Then we have:
\begin{align*}
\frac{f(x+tv) - f(x)}{t} &= \sum_{i=1}^n \left[ 
\frac{(x_i + tv_i)^{\frac{1}{3}} - x_i^{\frac{1}{3}}}{t}\cdot 
\frac{(x_i + tv_i)^{\frac{2}{3}} + (x_i + tv_i)^{\frac{1}{3}} x_i^{\frac{1}{3}} + x_i^{\frac{2}{3}}}{(x_i + tv_i)^{\frac{2}{3}} + (x_i + tv_i)^{\frac{1}{3}} x_i^{\frac{1}{3}} + x_i^{\frac{2}{3}}} 
\right] \\
&=\sum_{i=1}^n \left[
\frac{tv_i}{t}\cdot\frac{1}{(x_i + tv_i)^{\frac{2}{3}} + (x_i + tv_i)^{\frac{1}{3}} x_i^{\frac{1}{3}} + x_i^{\frac{2}{3}}}
\right] \\
&= \sum_{i=1}^n \left[
\frac{v_i}{(x_i + tv_i)^{\frac{2}{3}} + (x_i + tv_i)^{\frac{1}{3}} x_i^{\frac{1}{3}} + x_i^{\frac{2}{3}}}
\right].
\end{align*}
Letting \( y_i: = (x_i + tv_i)^{\frac{1}{3}} \) leads us to the quadratic equation
\[
y_i^2 + x_i^{\frac{1}{3}}y_i + x_i^{\frac{2}{3}} = 0\;\mbox{ for each }\;i=1,\ldots,n.
\]
Computing the discriminant for each \( i \) tells us that
\[
\Delta_i = x_i^{\frac{2}{3}} - 4x_i^{\frac{2}{3}} = -3x_i^{\frac{2}{3}} \leq 0,
\]
which ensures in turn that, whenever \( i = 1, \dots, n \), we get
\[
(x_i + tv_i)^{\frac{2}{3}} + (x_i + tv_i)^{\frac{1}{3}} x_i^{\frac{1}{3}} + x_i^{\frac{2}{3}} = y_i^2 + x_i^{\frac{1}{3}}y_i + x_i^{\frac{2}{3}} \geq 0.
\]
Therefore, for \( v_i \) sufficiently close to \( -\frac{1}{\sqrt{n}} \) and each \( i = 1, \dots, n \), it follows that
\[
\frac{f(x+tv) - f(x)}{t} = \sum_{i=1}^n \left[
\frac{v_i}{(x_i + tv_i)^{\frac{2}{3}} + (x_i + tv_i)^{\frac{1}{3}} x_i^{\frac{1}{3}} + x_i^{\frac{2}{3}}}
\right] \leq 0.
\]
The latter readily verifies that \( f \) in \eqref{f-ex1} is directionally Lipschitzian at $\ox$.
 \end{example}

The next example illustrates the statement of Proposition~\ref{cor:subgradients_sum} for the sum of a locally Lipschitzian and a directionally Lipschitzian functions when the qualification condition \eqref{eq:qualification_condition} is not trivial as in the fully Lipschitzian case while still holds and allows us to use the subdifferential sum rule in directionally Lipschitzian settings. 

\begin{example}\label{ex(H)} Consider the functions $f_1, f_2:\mathbb{R}\to\mathbb{R}$ given by
\[
f_1(x): = x^{\frac{1}{3}} \quad \text{and} \quad f_2(x): = \begin{cases}
\frac{1}{2}x \sin{\frac{1}{x}} & \text{if } x > 0, \\
-x & \text{if } x \leq 0.
\end{cases}
\]\noindent Since \( f_2 \) is locally Lipschitzian on $\mathbb{R}$ and \( f_1 \) is locally Lipschitzian for all $x\ne 0$, we get from the above that \( \partial^{\infty}f_2(x) = \{0\} \) for all \( x\in\mathbb{R} \) and \( \partial^{\infty}f_1(x) = \{0\} \) for all \( x\neq 0 \). On the other hand, it follows from definition \eqref{sin-sub} of the singular subdifferential that $ \partial^{\infty}f_1(0) = [0, \infty)$. It is easy to check that $f_2$ is directionally Lipschitzian at each point $x\in \mathbb{R}$ with the same unit vector as $f_1$. Furthermore, the qualifying condition \eqref{eq:qualification_condition} holds for these functions, and hence the subdifferential sum rule \eqref{eq:subdifferential_sum} is applicable in this case.
\end{example}
\begin{figure}[h]\label{fig1}
\centering
\begin{tikzpicture}
\begin{axis}[
axis lines=middle,
xlabel={$x$},
ylabel={$y$},
samples=100,
domain=-0.5:0.5,
ymin=-1.5, ymax=2.5,
width=10cm, height=7cm,
xtick={-2,-1,0,1,2},
ytick={-1,0,1,2},
legend pos=north west]
\addplot[blue, thick, domain=-2:2] {sign(x) * abs(x)^(1/3)};
\addplot[red, thick, domain=0.01:2] {0.5*x*sin(deg(1/x))};
\addplot[red, thick, domain=-2:0] {-x};
\end{axis}
\end{tikzpicture}
\caption{Graph of the functions $ f_1$ and $f_2$.}
\end{figure}
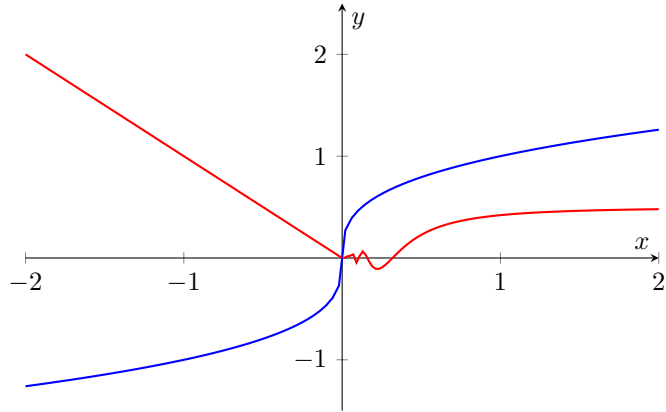

Below we present several properties of directionally Lipschitzian functions used in what follows. The first one is a  direct consequence of Definition~\ref{dir-lip}.

\begin{proposition}\label{constantdire}
If $f:\mathbb{R}^n\to\Bar{\mathbb{R}}$ is directionally Lipschitzian at $\ox\in\dom f$, then there exist $\delta>0$, a unit vector $u\in\mathbb{R}^n$, and a constant $L=L(\overline{x},u)$ dependent on $\ox$ and $u$ such that we have the estimate
\begin{align*}
\begin{array}{ll}
\qquad\qquad f(z+tv)-f(z)\leq t\cdot L(\overline{x},u)\\
\mbox{for all }\;v\in B_{\delta}(u), \; z\in B_{\delta}(\overline{x}), \;f(z)\in B_{\delta}(f(\overline{x})), \; t\in(0,\delta],
\end{array}
\end{align*}
where $B_r(\theta)$ denotes the open ball centered at a vector $\th$ with radius $r>0$.
\end{proposition}

The next proposition establishes {\em calculus rules} for the directional Lipschitzian functions telling us that the directional Lipschitzian property is preserved under the function summation and taking the maximum.

\begin{proposition}\label{max} Let $\psi,\phi: \mathbb{R}^n \to \mathbb{R}$ be directionally Lipschitzian functions at $\ox$ with the same unit vector $u\in\mathbb{R}^n$. Then the following hold:

{\bf(i)} The sum function $\varphi(x): = \psi(x)+\phi(x)$, $x\in\mathbb{R}^n$, is directionally Lipschitzian at $\ox$ with the unit vector $u$.

{\bf(ii)} The maximum function $\varphi(x): = \max\{\psi(x),\phi(x)\}$, $x\in\mathbb{R}^n$, is directionally Lipschitzian at $\ox$ with the unit vector $u$.
\end{proposition}
\begin{proof} To verify (i), take the common unit vector $u$ associated  with the functions $\psi$ and $\phi$. Then the directional Lipschitzian property of $\ph=\psi+\phi$ follows directly from \eqref{dir-lip1} due to the definition of the upper limit ``$\limsup$" therein. 

To prove (ii), observe the inequality
\begin{align}\label{eqmax}
\begin{array}{ll}
&\dfrac{\max\{\psi,\phi\}(x+tv)-\max\{\psi,\phi\}(x)}{t}\\\\
&\leq \dfrac{\max\{\psi(x+tv)-\psi(x),\phi(x+tv)-\phi(x)\}}{t}.
\end{array}
\end{align}
It follows from Proposition~\ref{constantdire} that \ there exist $\delta >0$ and $L(\overline{x},u)$ such that we get
\begin{align} \label{eqmax1}
\psi(x+tv)-\psi(x)\leq t\cdot L(\overline{x},u)~\mbox{ and }~\phi(x+tv)-\phi(x)\leq t\cdot L(\overline{x},u)
\end{align}
for all $v\in B_{\delta}(u), \; z\in B_{\delta}(\overline{x}), \;\psi(x)\in B_{\delta}(f(\overline{x})), \phi(x)\in B_{\delta}(f(\overline{x}))$, and  $t\in(0,\delta]$.
Combining \eqref{eqmax} and \eqref{eqmax1} gives us
$$
\displaystyle\limsup_{\substack{x\stackrel{\varphi}{\to} \overline{x} \\ v\to u \\ t\downarrow 0}}\bigg[\frac{\varphi(x+tv)-\varphi(x)}{t}\bigg]\leq L(\overline{x},u)<\infty,
$$
which therefore verifies the claimed directional Lipschitzian property by \eqref{dir-lip1}.
\end{proof}

\section{Regularized Models of Multiobjective Optimization}  \label{secMultiobectiveoptmin}

We start with recalling the notion {\em Pareto optimality/efficiency} in constrained problems of multiobjective optimization. 

\begin{definition}\label{pareto} Given a nonempty set $\Omega\subset \mathbb{R}^n$ and a vector-valued mapping $F=(f_{1},\ldots,f_{m}):\mathbb{R}^n\to \mathbb{R}^m$, we say that $\ox\in\O$ is a {\sc Pareto solution/minimizer} of the constrained multiobjective problem labeled as
\begin{equation}\label{eq:mp}
{\rm min}\big\{F(x)\;\big|\;x\in \Omega\big\}
\end{equation}
if there exists no point $x\in \Omega$ with $f_i(x)\leq f_i(\ox)$ for all
 $i\in \{1,\ldots,m\}$ and $f_j(x)<f_j(\ox)$ for at least one index $j \in \{1,\ldots,m\}$.
The notation 
$$
{\rm argmin}\big\{ F(x)\;\big|\;x\in \Omega\}
$$
stands for the collection of all Pareto optimal solutions to problem \eqref{eq:mp}. 
\end{definition}

The {\em local} Pareto solution of \eqref{eq:mp} is defined similarly, while we'll proceed with the study of (global) Pareto solutions for simplicity. Recall for comparison that the {\em weak Pareto} optimality in \eqref{eq:mp} means that there is no $x\in\O$ such that $f_i(x)<f_i(\ox)$ for all $i=1,\ldots,m$. We are not going to deal with the latter notion in this paper.

In fact, our main interest is devoted not to problem \eqref{eq:mp}, but its {\em regularizations}; see the discussion in Section~\ref{intro}. In this paper, we concentrate on the study of Pareto solutions to the {\em proximal-type regularization} of \eqref{eq:mp} written as
\begin{equation}\label{eq:mp1}
\mbox{min}\{\Psi(x):=F(x)+\lambda\|x-{\tilde{x}}\|^2\Upsilon\;\big|\;x\in \mathcal{D}\big\},
\end{equation}
where $\mathcal{D}:=\big\{x\ \in \Omega\;\big|\;\Phi(x):=F(x)-F({\tilde{x}})\le 0\}$ with the vector inequality understood componentwise, where $\tilde{x}\in \Omega$, $\lambda>0$, and $\Upsilon=(\upsilon_1,\ldots,\upsilon_m)\in\mathbb{R}^m$ is a unit vector with all the positive components. To the best of our knowledge, problem \eqref{eq:mp1} has its origins in \cite{Tikhonov1963} with stabilizing ill-posed problems (what is now known as the {\em Tikhonov regularization}), and in \cite{Moreau1965} with the {\em Moreau envelope} (or {\em Moreau-Yosida regularization}), which was subsequently used by Rockafellar to formalize the {\em proximal point algorithm} (PPA)  \cite{Rockafellar1976}. In the vector context, it was further developed in \cite{Bonnel2005} formalizing the vectorial PPA; see also \cite{bento2024refined,Bento2018,chuong2011hybrid} and the references therein for more recent developments in this direction.

In what follows, we study Pareto optimality in the regularized problem \eqref{eq:mp1}. Let us first verify the possibility of the {\em exact penalization} of geometric constraints in minimizing {\em directionally Lipschitzian} cost functions. The result below of its own interest is an extended directionally Lipschitzian counterpart of \cite[Proposition~2.4.3]{clarke} obtained therein for constrained optimization problems with Lipschitzian costs. 

\begin{theorem}\label{minindirectionally}
Let $f:\mathbb{R}^n\to {\mathbb{R}}$ be a directionally  1Lipschitzian and continuous function, and $\O\subset\mathbb{R}^n$ a nonempty closed set. If $\ox\in\O$ minimizes $f$ over $\O$, then there exist a unit vector $u\in\mathbb{R}^n$ and a constant $L=L(\ox,u)$ from Proposition~{\rm\ref{constantdire}} such that $\ox$ is a local minimizer of the penalized function $f+\tau d_{\Omega}$ whenever $\tau\geq L(\ox,u)$.
\end{theorem}
\begin{proof}
Suppose on the contrary that the result is false. Then there exists a sequence $\{x^k\}$ converging to $\ox$ as $k\to\infty$ such that
\begin{align}\label{lsc}
f(x^k)+\tau d_{\Omega}(x^k)< f(\ox).
\end{align}
By the imposed closedness of $\O$, we have $d_{\Omega}(x^k)>0$ since otherwise $x^k\in\Omega$, which contradicts the fact that $\ox$ is a minimizer of $f$ over $\Omega$. The assumed continuity of $f$ implies that $f(x^k)\to f(\ox)$ as $k\to\infty$. Let $k\in\N$ be sufficiently large to ensure that the directionally Lipschitzian constant holds in the ball around $x^k$ with radius $2d_{\Omega}(x^k)$; this is possible because $d_{\Omega}(x^k)$ converges to $0$ as $k\to\infty$. Therefore, it follows from \eqref{lsc} that we can choose  $\varepsilon\in (0,1)$ such that
\begin{equation}\label{lsc1}
f(x^k)+(1+\varepsilon) \tau d_{\Omega}(x^k)< f(\ox)
\end{equation}
and then select a positive number $t$ with $t\leq (1+\varepsilon)d_{\Omega}(x^k)$. Choosing further $v\in B_{\delta}(u)$ with $(x^k+tv)\in \Omega$ tells us that
\begin{align*}
f(x^k+tv)\leq f(x^k)+t\cdot L(\ox,u) \leq f(x^k)+t\tau
\end{align*}
by $t\leq (1+\varepsilon)d_{\Omega}(x^k)$. Thus it follows from
\eqref{lsc1} that $f(x^k+tv)<f(x^k)$, a contradiction, which completes the proof of the theorem.
\end{proof}

The next proposition is useful in the proof of the main result in Section~\ref{sec3}.

\begin{proposition}\label{lemma3.1} Let $\ox$ be a Pareto optimal solution to problem \eqref{eq:mp1}. Then for all $\gg>0$ and $x\in \Omega$, we have the condition
\begin{align}\label{max-func}
\varphi_{\gg}(x):=\max\big\{\psi_i(x)- \psi_i(\ox)+\gg, \phi_i(x)\;\big|\;i=1,\ldots,m\big\}>0,
\end{align}
where $\psi_i$ and $\phi_i$ are, respectively, the components of $\Psi$ and $\Phi$ in \eqref{eq:mp1}.
\end{proposition}
\begin{proof}
Supposing the contrary, we find $\gg_0>0$ and $x_1 \in \Omega$ such that
$\varphi_{\gg_0}(x_1)\leq 0$, which readily implies the inequalities
\begin{align*}
\psi_i(x_1)<\psi_i(x_1)+\gg_0\leq  \psi_i(\ox)\;\mbox{ and }\;
\phi_i(x_1)\leq 0\;\mbox{ for all }\; i=1,\ldots,m.
\end{align*}
This tells us that $x_1 \in \mathcal{D}$, which contradicts the Pareto optimality of  $\ox$ in \eqref{eq:mp1} and thus completes the proof of the proposition. 
\end{proof}

We conclude this section with an example of a multiobjective problem with two component functions $F=(f_1,f_2)$ one of which is directionally Lipschitzian but not locally Lipschitzian. However, Pareto optimal solutions to this problem belong to the region where $F$ is Lipschitz continuous. 

\begin{example} \label{exemplo3} Consider the mapping \( F = (f_1, f_2): \mathbb{R} \to \mathbb{R}^2 \) with the components  
\[
f_1(x): = (x+1)^2 \quad \text{and} \quad f_2(x): =
\begin{cases}
x^{\frac{1}{3}} & \text{if } x > 0, \\
x^2 + x & \text{if } x \leq 0.
\end{cases}
\]    
\noindent We obviously have that \( f_1 \) is locally Lipschitzian on the entire line $\mathbb{R}$, while \( f_2 \) is locally Lipschitzian for all \(x\neq 0\). Therefore,  \( \partial^{\infty}f_1(x) = \{0\} \) for all \( x\in\mathbb{R} \) and \( \partial^{\infty}f_2(x) = \{0\} \) for all \( x\neq 0 \). On the other hand, it follows from definition \eqref{sin-sub}  of the singular subdifferential of \(f_2\) at \(x=0\) that \( \partial^{\infty}f_2(0) = [0, \infty)\), i.e., $f_2$ is not locally Lipschitzian around $0$. It is easy to check nevertheless that $f_2$ is directionally Lipschitzian at $0$. The graphs of the functions \( f_1 \) and \( f_2 \) are shown in Figure~2. 
\begin{figure}[h]\label{fig2}
\centering
\begin{tikzpicture}
\begin{axis}[
axis lines=middle,
xlabel={$x$},
ylabel={$y$},
samples=100,
domain=-2:2,
ymin=-1.5, ymax=2.5,
width=10cm, height=7cm,
xtick={-2,-1,0,1,2},
ytick={-1,0,1,2},
legend pos=north west]
\addplot[blue, thick, domain=-2:0.5] {(x+1)^2};
\addplot[red, thick, domain=0:2] {x^(1/3)};
\addplot[red, thick, domain=-2:-0.01] {x^2 + x};
\end{axis}
\end{tikzpicture}
\caption{Graph of the functions $ f_1$ and $f_2$.}
\end{figure}
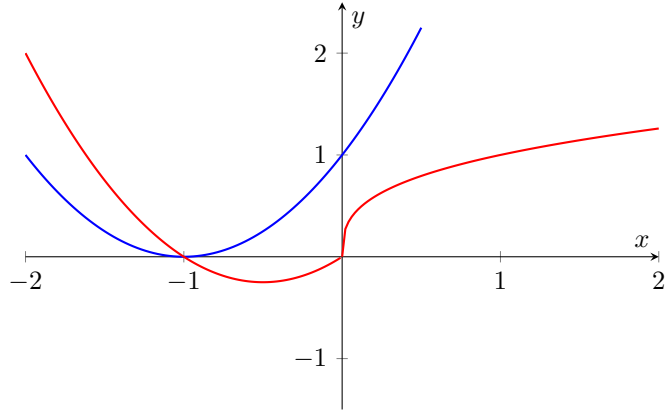

\noindent We can see that the set of Pareto optimal solutions to $F$ is  
\[
\operatorname{argmin} \{F(x): x\in\mathbb{R}\} = \left[-1, -\frac{1}{2}\right],
\]
and that the mapping \( F \) is locally Lipschitzian on this interval. 
\end{example}

\section{Necessary Conditions for Pareto Optimal Solutions}\label{sec3}

Now we are ready to establish novel necessary optimality conditions for Pareto optimal solutions to regularized multiobjective problems of type \eqref{eq:mp1} with directionally Lipschitzian components by using the limiting subgradients \eqref{lim-sub} of the functions involved. Although all the functions in \eqref{eq:mp1} are directionally Lipschitzian on their domains, they are assumed to be Lipschitz continuous around Pareto minimizers, which is the case, e.g., of the multiobjective problem in Example~\ref{exemplo3}. In Remark~\ref{non-lip}, we discuss another form of necessary conditions for Pareto optimal solutions in fully non-Lipschitzian multiobjective settings.

\begin{theorem}\label{teoCO} Let the mapping $F$ in problem \eqref{eq:mp1} be directionally Lipschitzian on $\mathbb{R}^n$  and Lipschitz continuous around a Pareto optimal solution $\ox$ to \eqref{eq:mp1}, and let $\O$ be closed. Then  there exist numbers $\tau>0$, $\alpha_i\geq 0$, and $\beta_i\geq 0$ for all $i=1,\ldots,m$  satisfying the following optimality conditions:
\begin{equation}\label{lagr}
0\in \sum_{i=1}^m(\alpha_i+\beta_i) \partial^M f_i(\ox)+\lambda\sum_{i=1}^m\alpha_i\upsilon_i(\ox-\tilde{x})+N_{\Omega}(\ox)\cap \tau\B,
\end{equation}
\begin{equation}
\nonumber\mbox{with}\quad\sum_{i=1}^m(\alpha_i+\beta_i)=1\quad\mbox{and}\quad \beta_i \big(f_i(\ox)-f_i(\tilde{x})\big)=0, \quad i=1,\ldots,m,
\end{equation}
where  $\tilde{x}$, $\lambda$, and $\Upsilon= (\upsilon_1, \ldots, \upsilon_m)$ are taken from the formulation of problem {\rm(\ref{eq:mp1})}.
\end{theorem} 
\begin{proof}
Fix a Pareto optimal solution $\ox$ to \eqref{eq:mp1} and pick a sequence of positive numbers $\gg_l\dn 0$ as $l\to\infty$. It follows from Proposition~\ref{lemma3.1} that
\begin{align}\label{pos}
\varphi_{\gg_l}(x)>0\;
\mbox{ whenever }\; l\in \mathbb{N}\;\mbox{ and }\;x\in \Omega
\end{align}
for the maximum function \eqref{max-func}. It follows from Proposition~\ref{max} that the function  $\varphi_{\gg_l}$ is directionally Lipschitzian on $\mathbb{R}^n$. We can also see that this function is locally Lipschitzian around $\ox$. Observe furthermore that for $0<\gg_l\leq 1$, the function $\ph_\gg(x)$ is nonnegative on $\O$, i.e., it is bounded from below on this set. Using \eqref{pos} and $\ox\in\mathcal{D}$ yields $\displaystyle\inf_{x\in \Omega} \varphi_{\gg_l}(x)\geq 0$, which tells us in turn that
\begin{align}
\nonumber \gg_l=\varphi_{\gg_l}(\ox)\leq \inf_{x\in \Omega} \varphi_{\gg_l}(x)+\gg_l.
\end{align}
Now we are in a position to apply the fundamental Ekeland variational principle (see, e.g., \cite[Theorem~2.26]{Morduk}) to the extended-real-valued l.s.c.\
function
\begin{equation*}
\Tilde\varphi_{\gg_l}(x):=\varphi_{\gg_l}(x)+\dd_{\O}(x),\quad x\in\O,
\end{equation*}
where $\dd_\O$ stands for the indicator function of $\O$ that equals 0 for $x\in\O$ and $\infty$ otherwise. This allows us to find $z^l\in\Omega$ such that
\begin{align}
\|\ox-z^l\|\leq \sqrt{\gg_l}\;\mbox{ and}\label{eq2.8}
\end{align}
\begin{align}
\nonumber \varphi_{\gg_l}(x)+\sqrt{\gg_l}\|x-z^l\|\geq \varphi_{\gg_l}(z^l)\;\mbox{  for all }\;x \in \Omega.
\end{align}
The latter mean that $z^l$ is a global solution to the constrained optimization problem
\begin{equation}\label{cost}
\min\;\varphi_{\gg_l}(x)+\sqrt{\gg_l}\|x-z^l\|,\quad  x\in\Omega.
\end{equation}
Note that the cost function in \eqref{cost} is directionally Lipschitzian by Proposition~\ref{max}(i) due to this property of $\varphi_{\gg_l}$ and the Lipschitz continuity of $\sqrt{\gg_l}\|\cdot-z^l\|$. Denoting by $L(z^l,u)$ a directionally Lipschitz constant of $\varphi_{\gg_l}(\cdot)+\sqrt{\gg_l}\|\cdot-z^l\|$ and employing Theorem~\ref{minindirectionally} ensure, for $l$ sufficiently large, that $z^l$ is a local minimizer of the unconstrained problem
\begin{align}\label{min-penal}
\min\,\chi_{\gg_l}(x):= \varphi_{\gg_l}(x)+\sqrt{\gg_l}\|x-z^l\|+\tau d_{\Omega}(x)
\end{align}
whenever $\tau\geq L(z^l,u)$. Thus $0 \in\partial^M \chi_{\gg_l}(z^l)$ by the generalized Fermat rule from \cite[Proposition~1.114]{Morduk}. Observing that, for $l$ sufficiently large, all the functions  in \eqref{min-penal} are locally Lipschitzian around $z^l$ allows us to apply the subdifferential sum rule from Proposition~\ref{cor:subgradients_sum}, where the qualification condition \eqref{eq:qualification_condition} holds automatically. We arrive in this way at 
\begin{align}
\nonumber 
0\subset\partial^M\varphi_{\gg_l}(z^l)+\sqrt{\gg_l}\partial^M(\|\cdot-z^l\|)(z^l)+\tau\partial^Md_{\Omega}(z^l).
\end{align}
Therefore, for each $l$ sufficiently large there exists $\omega^l\in\partial^M\varphi_{\gg_l}(z^l)$ satisfying the inclusion
\begin{align}
0 \in \omega^l+\sqrt{\gg_l}\partial^M(\|\cdot-z^l\|)(z^l)+\tau\partial^Md_{\Omega}(z^l).\label{eq2.9}
\end{align}
Applying to $\omega^l\in\partial^M\varphi_{\gg_l}(z^l)$ the subdifferentiation of the maximum function from \cite[Theorem~3.46(ii)]{Morduk}, we find
$\alpha^l\in \mathbb{R}^m_{+}$ and $\beta^l\in \mathbb{R}^m_{+}$ such that
\begin{equation}\label{eq2.10}
\sum_{i\in \mathcal{I}}(\alpha^l_i+\beta^l_i)=1,
\end{equation}
\begin{equation}\label{eq2.10a}
\begin{array}{ll}
&(\alpha^l_i+\beta^l_i)\big[(\psi(z^l)-\psi(\ox)+\gg_l) -\varphi_{\gg_l}(z^l)\big] =0~\mbox{and}\\
&(\alpha^l_i+\beta^l_i)\big[\phi(z^l)-\varphi_{\gg_l}(z^l)\big] =0\;\mbox{ for all }\;i=1,\ldots,m,
\end{array}
\end{equation}
\begin{equation}\label{eq2.11}
\omega^l \in \sum_{i=1}^m(\alpha_i^l+\beta_i^l) \partial^M f_i(z^l)+\sum_{i=1}^m\lambda\alpha_i^l\upsilon_i (z^l-\tilde{x}).
\end{equation}
Combining (\ref{eq2.9}) and (\ref{eq2.11}) leads us to the inclusion
\begin{equation*}
\begin{array}{ll}
0\in\disp\sum_{i=1}^m(\alpha_i^l+\beta_i^l)\partial^M f_i(z^l)+\disp\sum_{i=1}^m\lambda\alpha_i^l\upsilon_i (z^l-\tilde{x})\\
\qquad\displaystyle+\sqrt{\gg_l}\partial^M(\|\cdot-z^l\|)+\tau\partial^Md_{\Omega}(z^l).
\end{array}
\end{equation*}
In this way, we find sequences $\{u_i^l\}, \,  \{\tilde{u}^l\}$, and $ \{\tilde{v}^l\}$ with 
$$
u_i^l\in  \partial^M f_i(z^l),\;\tilde{u}^l\in \sqrt{\gg_l}\partial^M(\|\cdot-z^l\|)~\mbox{and}~\tilde{v}^l\in\tau\partial^Md_{\Omega}(z^l)
$$
satisfying the following equality for all $l\in\N$ sufficiently large:
\begin{align}
0= \sum_{i=1}^m(\alpha_i^l+\beta_i^l) u_i^l+\sum_{i=1}^m\lambda\alpha_i^l\upsilon_i (z^l-\tilde{x})+ \tilde{u}^l+ \tilde{v}^l.\label{eq2.12}
\end{align}
By (\ref{eq2.8}) and (\ref{eq2.10}), the sequences $\{\alpha_i^l\}$, $\{\beta_i^l\}$ as $i=1,\ldots,m$, and $\{z^l\}$ are bounded. Furthermore, the boundedness of the limiting subgradients of $\|\cdot-z^l\|$ and $d_{\Omega}(z^l)$ follows from the local Lipschitz continuity of these functions; see \cite[Corollary~1.81]{Morduk}. The boundedness of $\{u_i^l\}$ as $i=1,\ldots,m$ is due to the local Lipschitz continuity of $F$ around $\ox$ combined with the fact that $z^l \to \ox$ as $l\to\infty$. Passing to subsequences if necessary, we get that the sequences $\{\alpha_i^l\}$, $\{\beta_i^l\}$, $\{u_i^l\}$, $\{\tilde{u}^l\}$, and $\{\tilde{v}^l\}$ converge to $\alpha_i$, $\beta_i$, $u_i$, $ \tilde{u}$, and $ \tilde{v}$, respectively, where $\tilde{u}=0$ by $\gg_l \to 0$ as $l\to\infty$. These yield:
\begin{itemize}
\item $\alpha_i,\;\beta_i\geq 0$ for all $i=1,\ldots,m$ $\mbox{ and }\;$ $\disp\sum_{i=1}^m(\alpha_i+\beta_i)=1$ by \eqref{eq2.10}.

\item From \eqref{eq2.10a}, we obtain  $(\alpha_i+\beta_i)(\max\{0,f_i(\ox)-f_i(\tilde{x})\})=0$ and $$(\alpha_i+\beta_i)\big[(f_i(\ox)-f_i({\tilde{x}))-\max\{0,f_i(\ox)-f_i(\tilde{x}})\}\big]=0\;\mbox{ for all }\;i=1,\ldots,m.$$
Therefore, $(\alpha_i+\beta_i)(f_i(\ox)-f_i(\tilde{x}))=0$ and so $\beta_i(f_i(\ox)-f_i(\tilde{x}))=0$.

\item It follows from (\ref{eq2.12}) that
\begin{align}
\nonumber 0=\sum_{i=1}^m(\alpha_i+\beta_i)u_i+\lambda\sum_{i=1}^m\alpha_i\upsilon_i (\ox-\tilde{x}) +\tilde{v}.
\end{align}
The robustness property of Proposition~\ref{weakgraph} ensures that $u_i \in \partial^Mf_i(\ox)$ and $\tilde{v} \in \partial^Md_{\Omega}(\ox)$, which gives us the inclusion
\begin{align}
\nonumber 0\in\sum_{i\in \mathcal{I}}(\alpha_i+\beta_i) \partial^M f_i(\ox)+\lambda\sum_{i\in \mathcal{I}}\alpha_i\upsilon_i (\ox-\tilde{x})+\tau\partial^M d_{\Omega}(\ox)\;\mbox{ with some }\;\tau>0.
\end{align}
\end{itemize}
To verify \eqref{lagr}, it remains to use the subdifferential calculation \eqref{sub-dist} for the distance function $d_\O$ at $\ox\in\O$. This completes the proof of the theorem.
\end{proof}

\begin{remark}\label{non-lip}
{\rm It follows from the proof of Theorem~\ref{teoCO} that the Lipschitz continuity of the cost mapping $F=(f_1,\ldots,f_m)$ around the reference Pareto minimizer plays a significant role. This is due to the characterization $\partial^\infty f(\ox)=\{0\}$ of the local Lipschitz property, which supports the subgradient boundedness needed for passing to the limit. On the other hand, the {\em local Lipschitz continuity} can be {\em avoided} if the {\em limiting normal cone} \eqref{nor} to the epigraphs of $f_i$ is used instead of the subdifferential \eqref{lim-sub} of the components $f_i$; cf. \cite[Theorem~6.5]{mordukhovich2018variational} for the case of nondifferentiable programs with inequality constraints. Such a derivation is based on the {\em extremal principle} of variational analysis; see \cite{Morduk,Mordukh,mordukhovich2018variational}}.
\end{remark}

\section{Concluding Remarks and Future Research}\label{conc}

This papers investigates Pareto optimal solutions to a class of regularized multiobjective optimization problems with directionally Lipschitzian data subject to geometric constraints. Necessary optimality conditions are derived for such problems in terms of Mordukhovich's limiting subdifferential and normal cone provided that the Pareto minimizer in question belongs to the region where the vectorial cost mapping is locally Lipschitzian. 

Our future research pursues several issues. First we plan to consider Pareto solutions to fully non-Lipschitzian multiobjective optimization problems and derive for them necessary optimality conditions expressed in terms of limiting normals to epigraphs of the functions involved. Then we intend to study other regularization types, where the proximal term (squared norm) in \eqref{eq:mp1} is replaced by some nonsmooth regularizators. They include the nonsmooth convex $\ell^1$-norm regularizators typical in problems of machine learning and statistics (like the Lasso problem \cite{lasso}) and nonsmooth nonconvex $\ell^0$-norm as in models of cancer research \cite{proton}.

\begin{acknowledgement}
{\rm Research of Boris S. Mordukhovich was partly supported by the US National Science Foundation under grant DMS-2204519 and by the Australian Research Council under Discovery Project DP250101112, the research of Glaydston de C. Bento was supported in part by CNPq grants 314106/2020-0, the research of Jo\~ao Xavier da Cruz Neto was partly supported by CNPq grants 302156/2022-4 and the research of Jurandir de Oliveira Lopes was partly supported by CNPq grants 305415/2025-5.}
\end{acknowledgement}


\begin{thebibliography}{10}

\bibitem{bao} T. Q. Bao and B. S. Mordukhovich, Relative Pareto minimizers for multiobjective problems: Existence and optimality conditions, {\em Math. Program.} {\bf 122} (2010), 301--347.

\bibitem{bento2024refined}
G. C.~Bento, J. X.~Cruz~Neto, J.~Lopes, B. S.~Mordukhovich, and P.~Silva~Filho,
\newblock A refined proximal algorithm for nonconvex multiobjective
optimization in Hilbert spaces,
\newblock {\em J. Global Optim.} {\bf 92} (2025), 187--203.

\bibitem{Bento2018}
G.~C. Bento, J.~X. Cruz~Neto, G.~L\'opez, A.~Soubeyran, and J.~C.~O. Souza,
\newblock The proximal point method for locally Lipschitz functions in
multiobjective optimization with application to the compromise problem.
\newblock {\em SIAM J. Optim.} {\bf 28} (2018), 1104--1120.

\bibitem{bento2018proximal}
G.~C. Bento, J.~X. Cruz~Neto, and L.~V. de~Meireles,
\newblock Proximal point method for locally Lipschitz functions in
multiobjective optimization of Hadamard manifolds,
\newblock {\em J. Optim. Theory Appl.} {\bf 179} (2018), 37--52.

\bibitem{bento2022pareto}
G.~C. Bento, J.~X. Cruz~Neto, L.~V. Meireles, and A.~Soubeyran,
\newblock Pareto solutions as limits of collective traps: An inexact
multiobjective proximal point algorithm,
\newblock {\em Ann. Oper. Res.} {\bf 316} (2022), 1425--1443.

\bibitem{Bonnel2005}
H.~Bonnel, A.~N. Iusem, and B.~F. Svaiter,
\newblock Proximal methods in vector optimization.
\newblock {\em SIAM J. Optim.} {\bf 15} (2005), 953--970.

\bibitem{burke2021}
J.~V. Burke and Q.~Lin,
\newblock Convergence of the gradient sampling algorithm on directionally
Lipschitz functions,
\newblock {\em Set-Valued Var. Anal.} {\bf 29} (2021), 949--966.

\bibitem{chuong2011hybrid}
T.~D. Chuong, B.~S. Mordukhovich, and J.-C. Yao,
\newblock Hybrid approximate proximal algorithms for efficient solutions in
vector optimization,
\newblock {\em J. Nonlinear Convex Anal.} {\bf 12} (2011), 257--286.

\bibitem{clarke}
F.~H. Clarke, {\em Optimization and Nonsmooth Analysis}, Wiley, New York, 1983.

\bibitem{proton} X. Cong, X. Ding, B. S. Mordukhovich, A. Nguyen, and L. Zhao, $\ell_0$-Norm multiobjective optimization models motivated by applications to proton therapy, to appear in J. Nonlin. Var. Anal. (2026); arXiv:2509.06833.

\bibitem{gal1997}
T.~Gal and T.~Hanne,
\newblock On the development and future aspects of vector optimization and MCDM: A tutorial,
\newblock In {\em Multicriteria Analysis: Proceedings of the XIth International
Conference on MCDM, 1--6 August 1994, Coimbra, Portugal}, pp.\ 130--145.
Springer, Berlin, 1997.

\bibitem{jahn2009vector}
J.~Jahn,
\newblock {\em Vector Optimization: Theory, Applications, and Extensions}, 
\newblock Springer, Berlin, 2011.

\bibitem{ktz} A. A. Khan, C. Tammer, and C. Z\u alinescu, {\em Set-Valued Optimization: An Introduction and Applications}, Springer, Berlin, 2015.

\bibitem{kobis} E. K\"obis and C. Tammer, Optimality conditions in optimization under uncertainty, {\em J. Nonlinear Var. Anal.} {\bf 7} (2023), 769--784.

\bibitem{minami}
M.~Minami,
\newblock Weak Pareto-optimal necessary conditions in a nondifferentiable
multiobjective program on a Banach space,
\newblock {\em J. Optim. Theory Appl.} {\bf 41} (1983), 451--461.
  
\bibitem{Morduk}
B.~S. Mordukhovich,
\newblock {\em Variational Analysis and Generalized Differentiation I: Basic
Theory}, Springer, Berlin, 2006.

\bibitem{Mordukh}
B.~S. Mordukhovich,
\newblock {\em Variational Analysis and Generalized Differentiation II: Applications}, Springer, Berlin, 2006.

\bibitem{mordukhovich2018variational}
B.~S. Mordukhovich,
\newblock {\em Variational Analysis and Applications},
\newblock Springer, Cham, Switzerland, 2018.

\bibitem{Moreau1965}
J.-J. Moreau,
\newblock Proximit{\'e} et dualit{\'e} dans un espace hilbertien.
\newblock {\em Bull. Soc. Math. Fr.} {\bf 93} (1965), 273--299.

\bibitem{Rockafellar1976}
R.~T. Rockafellar,
\newblock Monotone operators and the proximal point algorithm.
\newblock {\em SIAM J. Control Optim.} {\bf 14} (1976), 877--898.

\bibitem{Roc79}
R.~T. Rockafellar,
\newblock Directionally Lipschitzian functions and subdifferential calculus.
\newblock {\em Proc. London Math. Soc.} {\bf 3} (1979). 331--355.

\bibitem{rw}
R.~T. Rockafellar and R. J-B. Wets,
\newblock {\em Variational Analysis}, Springer, Berlin, 1998.

\bibitem{Tikhonov1963}
A.~N. Tikhonov and V.~Y. Arsenin,
\newblock {\em Solutions of Ill-posed Problems},
\newblock Winston \& Sons, 1977.

\bibitem{lasso} R. Tibshirani, Regression shrinkage and selection via the Lasso. {\em J. R. Stat. Soc.} {\bf 58} (1996), 267--288.

\bibitem{white1990}
D.~White.
\newblock A bibliography on the applications of mathematical programming
multiple-objective methods,
\newblock {\em J. Oper. Res. Soc.} {\bf 41} (1990), 669--691.

\end{thebibliography}
\end{document}